\documentclass[12pt,a4paper]{amsart}
\usepackage{amsmath,amssymb,amsthm}
\usepackage{enumitem}
\usepackage{ifthen,verbatim}
\usepackage{mathrsfs}
\usepackage{color}
\usepackage{url}
\usepackage[english]{babel}
\usepackage{here}
\usepackage[T1]{fontenc}
\usepackage{floatflt,graphicx,graphics}
\usepackage{a4wide}
\usepackage{cite}
\usepackage{ifthen}

\theoremstyle{plain}
\newtheorem{theorem}{Theorem}
\newtheorem{lemma}{Lemma}

\newtheorem{conjecture}{Conjecture}

\theoremstyle{definition}

\newtheorem{remark}{Remark}

\nonstopmode
\begin{document}
\title[Asymptotics of the exterior conformal modulus]{Asymptotics of the exterior conformal modulus of an arbitrary quadrilateral under stretching map}

\author{Semen~R.~Nasyrov}
\address{Kazan Federal University, Kremlyovskaya str. 35, Tatarstan, 420008,
Russian Federation}
\email{snasyrov@kpfu.ru}

\author{Giang~V.~Nguyen}
\address{Kazan Federal University, Kremlyovskaya str. 35, Tatarstan, 420008,
Russian Federation}
\email{nvgiang.math@gmail.com}

\begin{abstract}
In this paper, we focus on studying the distortion of the exterior conformal modulus of a quadrilateral of sufficiently arbitrary form under the stretching map along the abscissa axis with coefficient $H\to\infty$. By using the properties of quasiconformal transformations and taking into account some facts from the theory of elliptic integrals, we confirm that the asymptotic behavior of this modulus does not depend on the shape of the boundary of the quadrilateral. Especially, it is equivalent to $(1/\pi)\log H$ as $H\to\infty$. Therefore, we give a solution to the Vuorinen problem for the exterior region of an arbitrary quadrilateral.
\end{abstract}

\keywords{\it quadrilateral, conformal modulus, exterior conformal modulus, quasiconformal mapping, convergence of domains to a kernel.}

\maketitle

\section{Introduction}
Let $Q$ be a Jordan domain and $z_j\, (1\le j\le 4)$ be four fixed points on the boundary of $Q$, which are arranged such that the increasing of the index $k$ corresponds to the positive direction
when $z_j$ runs along the boundary $\partial Q$. Then $\mbox{\boldmath$Q$}=(Q; z_1, z_2, z_3, z_4)$ is called a quadrilateral with the domain $Q$, and the vertices $z_j\, (1\le j\le 4)$. With the convention $z_5\equiv z_1$, we will define the arcs $(z_k, z_{k+1})$ of the boundary $\partial Q$ as the edges of the quadrilateral.

For a given quadrilateral $\mbox{\boldmath$Q$}=(Q; z_1, z_2, z_3, z_4)$, let $\Gamma$ be the family of all curves in $Q$ connecting $(z_1, z_2)$ and $(z_3, z_4)$. The interior conformal modulus $\operatorname{Mod}(\mbox{\boldmath$Q$})$ of
$\mbox{\boldmath$Q$}$ is the extremal length $\lambda(\Gamma)$ of $\Gamma$ (for the concept of extremal length see, e.g., \cite[Chap.~1, Sect.~D]{ahlfors2006lectures}, \cite[Chap.~3,
Sect.~11]{kuhnau2005conformal}). 

\begin{remark} The interior conformal modulus of a quadrilateral $\mbox{\boldmath$Q$}= (Q; z_1, z_2, z_3, z_4)$ can also be defined similarly in the case when $Q$ is not a Jordan domain, but only a simply connected domain with the nondegenerate boundary, and $z_1$, $z_2$, $z_3$, $z_4$ are four different boundary prime ends of $Q$ (see, e.g., \cite[Chap.~2, Sect.~3]{goluzin1969geometric}).
\end{remark}

Now we consider the family of all curves $\widetilde{\Gamma}$,  joining the edges $(z_1, z_2)$ and $(z_3, z_4)$ in the complement of the domain $Q$ to the extended complex plane. Then the extremal length $\lambda(\widetilde{\Gamma})$ of $\widetilde{\Gamma}$, is called the exterior conformal modulus of the quadrilateral $\mbox{\boldmath$Q$}$, which will be denoted by
$\operatorname{ExtMod}(\mbox{\boldmath$Q$})$. It is clear that if we denote $\mbox{\boldmath$Q$}^c=(Q^c; z_4, z_3, z_2, z_1)$, where $Q^c=\overline{\mathbb{C}}\setminus Q$ is the complement of $Q$ in the extended complex plane, then $\operatorname{Mod}(\mbox{\boldmath$Q$}^c)=\operatorname{ExtMod}(\mbox{\boldmath$Q$})$.

We note that the interior conformal modululs of the quadrilateral is also closely connected with the Dirichlet integral (see, e.g., \cite{dubinin2014condenser}). For a given quadrilateral $\mbox{\boldmath$Q$}=(Q;
z_1, z_2, z_3, z_4)$ with the piecewise smooth boundary and a function $u$ that is harmonic in the domain $Q$, continuous in its closure and satisfying the Dirichlet-Neumann boundary conditions  $u(z) = 0$ for $z\in (z_2, z_3)$ and $u = 1$ for
$z\in (z_4, z_1)$, $\frac{\partial u}{\partial n}(z)=0$ for $z\in (z_1, z_2)\cup (z_3, z_4)$, where $n$ is the direction of the outward normal to the boundary $Q$, the interior conformal modulus of the quadrilateral $\mbox{\boldmath$Q$}$ is equal to the Dirichlet integral
$$\operatorname{Mod}(\mbox{\boldmath$Q$})%=\int\limits_{Q}|\nabla u|^2dm
=\int\limits_{Q}\left\{\left(\frac{\partial u}{\partial x}\right)^2+\left(\frac{\partial u}{\partial y}\right)^2\right\}dxdy.$$
The exterior conformal modulus $\operatorname{ExtMod}(Q)$ is equal to the Dirichlet integral of a function that is harmonic in the complement $Q^c$ of the domain $Q$ and continuous in its closure, with similar boundary conditions.

Finally,  the interior conformal modulus can be defined in terms of conformal mappings.  If $f : Q\to [0, 1]\times[0, m]$ is a conformal mapping of the domain $Q$ onto the rectangle such that the vertices of $Q$ go to the vertices of the rectangle, and $f(z_1) = 0,$ then $\operatorname{Mod}(\mbox{\boldmath$Q$}) = m$; for the exterior modulus, we must take a conformal mapping onto the exterior rectangle $Q^c$ of the domain $Q$.

The interior conformal modulus of a quadrilateral is invariant under conformal transformations. It is also quasiinvariant under quasiconformal mappings (see, e.g., \cite[Chap.~2,
Sect.~A]{ahlfors2006lectures}, \cite[Chap.~3, Sect.~22]{kuhnau2005conformal}): if $f_{\mathcal{K}}$ is an $\mathcal{K}$-quasiconformal mapping, then for any quadrilateral $\mbox{\boldmath$Q$}$ we have the inequality
\begin{align}\label{eq:quasi_prop}
\dfrac{1}{\mathcal{K}}\operatorname{Mod}(\mbox{\boldmath$Q$})\leq \operatorname{Mod}(f_{\mathcal{K}}(\mbox{\boldmath$Q$}))\leq \mathcal{K} \operatorname{Mod}(\mbox{\boldmath$Q$}).
\end{align}
Here $f_{\mathcal{K}}(\mbox{\boldmath$Q$})$ is a quadrilateral whose domain and vertices are the images of the domain and vertices of the quadrilateral $\mbox{\boldmath$Q$}$ under the mapping $f_{\mathcal{K}}$.

\begin{remark} Obviously, the analogue of the double inequality \eqref{eq:quasi_prop} is still valid for the exterior conformal modulus.
\end{remark}

In 2005, Prof. M.~Vuorinen stated the problem of investigating the distortion of the conformal moduli of domains under the action of the $H$-quasiconformal mapping $f_H$ which defined as follows
\begin{equation}\label{eq:stretch_map}
f_H:x+iy\mapsto Hx+iy\quad (H>1).
\end{equation}
Especially,  he was interested in what is the behavior of the corresponding conformal moduli as $H\to\infty$ ? In this case, it is of interest to study the distortion of conformal moduli of both quadrilaterals and doubly connected domains.

Many works have been devoted to the study of the Vuorinen problem. In particular, \cite{dautova2018, dautova2019, nasyrov2015riemann, nvgiang2021, nguyen2022}) studied the distortion of the interior conformal modulus of an arbitrary bounded or unbounded doubly connected domain. In addition, in \cite{nvgiang2021, nguyen2022} some questions related to the asymptotic behavior of the exterior conformal modulus of a quadrilateral has been also studied.

Investigating the distortion of the behavior of the exterior conformal modulus of a quadrilateral under the action of $f_H$ is really worthy of attention.  W.~Bickley \cite{bickley} and, later,
P.~Duren and J.~Pfaltzgraff \cite{duren1993robin}  suggested some formulas connecting the values of the interior and exterior conformal moduli of the rectangle. By 2013, using the results of
\cite{duren1993robin}, M.~Vuorinen and H.~Zhang \cite[Theorem 4.3]{vuorinen2013exterior} established some upper and lower bounds for the exterior conformal modulus of the rectangle with a fixed interior modulus. 

In connection with the study of the exterior conformal moduli presented in \cite{vuorinen2013exterior, nasyrov2021moduli}, the first author of this paper gave a conjecture about the asymptotics of the exterior conformal modulus under the action of the extension $f_H$ as $H\to\infty$.  Let us recall its content.

Let $f$ and $g$ be two continuous functions on the segment $[a,b]$, $-\infty< a<b<+\infty$,  such that $f(x)<g(x)$ for all $x\in[a,b]$.  Put $\mbox{\boldmath$Q$}= (Q; z_1, z_2, z_3, z_4)$, where $Q$ is bounded by two vertical segments $[z_1, z_2]$ and $[z_3, z_4]$ with endpoints $z_1=a+ig(a) $, $z_2=a+if(a)$, $z_3=b+if(b)$, $z_4=b+ig(b)$, and two curves are graphs of functions $f, g$:
$$
\mathcal{G}_1 = \{x + iy\,:\,y = f(x),\;a \leq x \leq b\},\quad \mathcal{G}_2 = \{x + iy\,:\,y = g(x),\;a \leq x \leq b\}.
$$
In what follows, the set of all quadrilaterals of this form (i.e., quadrilaterals corresponding to arbitrary segments $[a,b]$ and functions $f$ and $g$ satisfying the restrictions described above) will be denoted by $\mathfrak{S}$.

Denote by $\mbox{\boldmath$Q$}_H$ the quadrilateral which is the image of $\mbox{\boldmath$Q$}$ under the mapping~$f_H$; 
we mean that the domain and vertices of $\mbox{\boldmath$Q$}$  under the mapping $f_H$, go to the domain and vertices of the quadrilateral $\mbox{\boldmath$Q$}_H$.  Prof. S.~R.~Nasyrov predicted that
\begin{conjecture}[S.~R.~Nasyrov]\label{conject_nasyrov}
The asymptotic behavior of the exterior conformal modulus of $\mbox{\boldmath$Q$}_H$ does not depend on the shape of the boundary of $Q$ as $H\to\infty$. Moreover, $\operatorname{ExtMod}(\mbox{\boldmath$Q$}_H)\sim (1/\pi)\log H$ as $H\to\infty$.
\end{conjecture}

Recently in \cite[Theorem~1.3]{nguyen2023} A.~Duytin and the second author of this article have been confirmed the validity of the conjecture ~\ref{conject_nasyrov} for any quadrilaterals symmetric about both coordinate axes.

In this article, we prove the validity of the conjecture~\ref{conject_nasyrov} for the case of arbitrary quadrilaterals of class $\mathfrak{S}$. The following theorem is the main result of the paper.
\begin{theorem}\label{thr:main}
If $\mbox{\bf Q}\in\mathfrak{S}$ then
\begin{equation}\label{eq:main}
 \operatorname{ExtMod}(\mbox{\boldmath$Q$}_H)\sim \frac{1}{\pi}\,\log H \quad\text{as }\quad H\to\infty.
\end{equation}
\end{theorem}

Note that the theorem \ref{thr:main} shows that the character of the asymptotic behavior of the exterior conformal modulus $\operatorname{ExtMod}(\mbox{\boldmath$Q$}_H)$ of the quadrilateral $\mbox{\boldmath$Q$} _H\in \mathfrak{S}$ is significantly different from the situation in the case of the interior conformal modulus $\operatorname{Mod}(\mbox{\boldmath$Q$}_H)$. The asymptotic behavior of $\operatorname{Mod}(\mbox{\boldmath$Q$}_H)$ depends on the shape of the boundary curves $\mathcal{G}_1$ and $\mathcal{G}_1$. As shown in \cite[Theorem~4.2]{nvgiang2021},
$$
\operatorname{Mod}(\mbox{\boldmath$Q$}_H)\sim cH, \quad \mbox{ where }\quad c=\int\limits_a^b\frac{dx}{g(x)-f (x)}.
$$
%%%%%%%%%%%%%%%%%%%%%%%%%%%
\section{Monotonous properties of conformal moduli}
In this section, we present some properties that describe the behavior of a conformal modulus when the quadrilateral changes.

Let $\Gamma_1, \Gamma_2$ be two families of curves, we write $\Gamma_1<\Gamma_2$ if every curve $\gamma_2\in\Gamma_2$
contains some $\gamma_1\in\Gamma_1$ as its subarc, i.e. $\Gamma_1\supset\Gamma_2$.

The following theorem is evident, which describes the change in the extremal length $\lambda(\Gamma)$ of a family of curves $\Gamma$ (see, e.g.
\cite[Chap.~1, Sect.~D]{ahlfors2006lectures}).

\begin{theorem}\label{thr:comp1} If $\Gamma_1<\Gamma_2$, then $\lambda(\Gamma_1)<\lambda(\Gamma_2)$. 
\end{theorem}

From Theorem~\ref{thr:comp1} we can immediately obtain some
monotonicity properties of the conformal modulus of the quadrilateral (see, e.g., \cite[Chap.~2, Sect.~3]{papamichael2010numerical}).

\begin{theorem}\label{thr:comp2} Let $\mbox{\boldmath$Q$}=(Q;z_1,z_2,z_3,z_4)$ and
$\mbox{\boldmath$Q$}'=(Q';z'_1,z'_2,z'_3,z'_4)$ be two quadrilaterals.

$(i)$ If $Q=Q'$,  the vertices $z'_1$ and $z'_2$ are on the
edge $(z_1, z_2)$, and the vertices $z'_3$ and $z'_4$ are on the edge
$(z_3, z_4)$ of $\mbox{\boldmath$Q$}$, then $\operatorname{Mod}(\mbox{\boldmath$Q$})\le
\operatorname{Mod}(\mbox{\boldmath$Q$}')$.

$(ii)$ If $Q\supset Q'$, $z_k=z'_k$, $1\le k\le 4$, and
two the quadrilaterals have the same pair of opposite edges,
$(z_1, z_2)$ and $(z_3, z_4)$, then $\operatorname{Mod}(\mbox{\boldmath$Q$})\le
\operatorname{Mod}(\mbox{\boldmath$Q$}')$.
\end{theorem}

We can say that, in the case $(i)$ $\mbox{\boldmath$Q$}'$
is obtained from $\mbox{\boldmath$Q$}$ by narrowing its original
edges, and in the case of $(ii)$ $\mbox{\boldmath$Q$}'$ is obtained from $\mbox{\boldmath$Q$}$ pressing its sides. Theorem~\ref{thr:comp2} states that under these transformations the conformal modulus of the quadrilateral increases.

\section{Some properties of elliptic integrals}
The incomplete elliptic integrals of the first and second kinds are defined by the equations, respectively (see, e.g., \cite{byrd1971handbook}).
\begin{align}\label{incom_form_int1}
F(z, k)= \int\limits_0^z \frac{dt}{\sqrt{(1 - t^2)(1 - k^2 t^2)}},
\end{align}
\begin{align}\label{incom_form_int2}
E(z, k)= \int\limits_0^z \sqrt{\frac{1 - k^2 t^2}{1 - t^2}}\,dt,
\end{align}
where $k \in (0, 1)$ is a real parameter.

We can represent some terms in the integrands of the elliptic integrals as binomial series. The respective incomplete elliptic integrals can be written in the form (see, e.g.,
\cite[formulas~902.00--903.00, p.~300]{byrd1971handbook}).
\begin{align}\label{eq:exp_int1}
F(x, k)=\sum\limits_{n=0}^{\infty}\frac{(2n)!}{2^{2n}{n!}^2}k^{2n}I_{2n}(x),
\end{align}
\begin{align}\label{eq:exp_int2}
E(x, k)=-\sum\limits_{n=0}^{\infty}\frac{(2n)!}{2^{2n}n!^2(2n-1)}k^{2n}I_{2n}(x).
\end{align}
Here
$$
I_{2n}(x)=\int\limits_0^x\frac{t^{2n}}{\sqrt{1-t^2}}dt
$$
is the integral which can be computed by using the following recursion formula:
$$
I_0(x)=\arcsin x, \quad I_2(x)=\frac{1}{2}\left(\arcsin x-x\sqrt{1-x^2}\right),
$$
$$
I_{2n}(x)=\frac{2n-1}{2n}I_{2(n-1)}(x)-\frac{1}{2n}x^{2n-1}\sqrt{1-x^2}, \quad n\ge 1.
$$
From this it follows that
\begin{equation}\label{I02}
I_0(x)-2I_2(x)=x\sqrt{1-x^2}>0, \quad 0<x<1.
\end{equation}
Moreover,
\begin{equation}\label{Ipi}
 0<I_{2n}(x)\le I_{0}(1)=\frac{\pi}{2}, \quad n\ge 0, \quad 0<x<1.
\end{equation}

The complete elliptic integrals are a special case of the incomplete elliptic integrals. Putting  $z=1$ in \eqref{incom_form_int1} and \eqref{incom_form_int2},  we obtain the following integrals
\begin{align*}
K(k)= \int\limits_0^1 \frac{dt}{\sqrt{(1 - t^2)(1 - k^2 t^2)}}, \quad E(k)= \int\limits_0^1 \sqrt{\frac{1 - k^2 t^2}{1 - t^2}}\,dt, \quad\text{where}\quad k \in (0, 1)
\end{align*}
which are called, respectively, the complete elliptic integrals of the first and second kinds. We will denote $K'(k) = K(k')$ and $E'(k) = E(k')$, where $k':= \sqrt{1 - k^2}$.

On the other hand, for small $k$, the complete elliptic integrals can be expressed by the formulas
\begin{align}\label{eq:com_int1_serie1}
K(k)=\frac{\pi}{2}\sum\limits_{n=0}^{\infty}\left[\frac{(2n)!}{2^{2n}n!^2}\right]^2k^{2n},
\end{align}
\begin{align}\label{eq:com_int2_serie1}
E(k)=-\frac{\pi}{2}\sum\limits_{n=0}^{\infty}\left[\frac{(2n)!}{2^{2n}n!^2}\right]^2\frac{k^{2n}}{(2n-1)},
\end{align}
and for large $k$, the complete elliptic integrals are expressed in the forms:
\begin{align}\label{eq:com_int1_serie2}
K(k)=\frac{2}{\pi}K'(k)\ln\frac{4}{k'}-\sum\limits_{n=1}^{\infty}\left[\frac{(2n)!}{2^{2n}n!^2}\right]^2\sum\limits_{m=1}^n\frac{(k')^{2n}}{m(2m-1)},
\end{align}
\begin{multline}\label{eq:com_int2_serie2}
E(k)=\frac{2}{\pi}[K'(k)-E'(k)]\ln\frac{4}{k'}-\sum\limits_{n=1}^{\infty}\left[\frac{(2n)!}{2^{2n}n!^2}\right]^2\frac{2n}{(2n-1)}\sum\limits_{m=1}^n\frac{(k')^{2n}}{m(2m-1)}\\+\sum\limits_{n=0}^{\infty}\left[\frac{(2n)!}{2^{2n}n!^2}\right]^2\frac{(k')^{2n}}{(2n-1)^2}.
\end{multline}

From \eqref{eq:com_int1_serie1}, \eqref{eq:com_int2_serie1}, \eqref{eq:com_int1_serie2} and \eqref{eq:com_int2_serie2}, the following equalities are verified:
\begin{equation}\label{asymke0}
\lim_{k \to 0} K(k) = \frac{\pi}{2}\,,\;\;\;\lim_{k \to 0} E(k) = \frac{\pi}{2}\,,\;\;\;\lim_{k \to 0} \left(K'(k) - \text{ln}\,\frac{4}{k} \right) = 0,\;\;\;\lim_{k \to 0} E'(k)= 1,
\end{equation}
\begin{equation}\label{asymke1}
\lim_{k \to 1} \left(K(k) - \text{ln}\,\frac{4}{k'} \right) = 0,\;\;\;\lim_{k \to 1} E(k) = 1,\;\;\;\lim_{k \to 1} K' (k) = \frac{\pi}{2}\,,\;\;\;\lim_{k \to 1} E'(k) = \frac{\pi}{2}.\end{equation}
%%%%%%%%%%%%%%%%%%%%%%%%%%%%%%%%%
%%%%%%%%%%%%%%%%%%%%%%%%%%%%%%%%%%
\section{Proof of the main result}
We consider the quadrilateral $\mbox{\boldmath$Q$}=(Q; A, B, C, D)\in\mathfrak{S}$ with vertices $A=a+if(a)$, $B=b+if(b)$, $C=b+i g(b)$, and $D=a+ig(a)$. Let
$$Q_H=f_H(Q)=\{(x, y)\in\mathbb{C}\,:\, aH\le x\le bH, f(x/H)\le y\le g(x/H)\}$$
be the image of the domain $Q$ under stretching map $f_H$  defined by \eqref{eq:stretch_map}, $A_H=f_H(A)$,  $B_H=f_H(B)$,  $C_H=f_H(C)$,  $D_H=f_H(D)$. Then $\mbox{\boldmath$Q$}_H=(Q_H; A_H, B_H, C_H, D_H)$ is the image of $\mbox{\boldmath$Q$}$ under this mapping.

Let $I_H=aH+i(f(a)+g(a))/2$ and $J_H=bH+i(f(b)+g(b))/2$ be the midpoints of the segments $[A_H, D_H]$ and $[B_H, C_H]$.  For the natural number $n$, let $a\equiv x_1<x_2<\ldots<x_{n+1}\equiv b$ be some points of $[a, b]$ and $y_j:=( f(x_j )+g(x_j))/2$, $1\le j\le n+1$. Clearly, $f(x_j)<y_j<g(x_j)$ for all $1\le j\le n+1$. Denote $X_{jH}=x_jH+i\,y_j$. Obviously, it is possible to fix the values of $x_j$ so that the $n$-polyline $\gamma_{H}$ connecting the points $X_{1H}, \ldots, X_{(n+1)H} $ lies in domain $Q_H$ and separates the graphs of $f(x/H)$ and $g(x/H)$, $x\in[aH,bH]$.  For every $j=1,
\ldots, n$, the equation of the $j$-th segment $[X_{jH}, X_{(j+1)H}]$ of $\gamma_{H}$ is
$$
y=\frac{a_j}{H}x+b_j, \quad\text{ if } \quad x_jH\le x\le x_{j+1}H,
$$
where
$$
a_j=\frac{y_{j+1}-y_j}{x_{j+1}-x_j}\quad\text{ and }\quad b_j=-\frac{x_jy_{j+1}-x_{j+1}y_j}{x_{j+1}-x_j}.
$$

Consider the mapping $\varrho_H\,:\,x+iy\mapsto x+iv(x, y)$, where
$$v(x, y)=\begin{cases}
y-\frac{f(a)+g(a)}{2},&\text{ if }\quad x\le aH,\\
y-\frac{a_j}{H}x-b_j,&\text{ if }\quad x_jH\le x\le x_{j+1}H, \quad 1\le j\le n,\\
y-\frac{f(b)+g(b)}{2},&\text{ if }\quad x\ge bH.
\end{cases}
$$
It is a piecewise-linear homeomorphism of the complex plane. It is not difficult to verify that $\varrho_H$ is an $\mathcal{K}_H$-quasiconformal mapping from the $z$-plane onto the $w$-plane,
with the coefficient
$$\mathcal{K}_H=\frac{1+\kappa_H}{1-\kappa_H},$$
where $\kappa_H=\frac{\mathfrak{a}}{\sqrt{\mathfrak{a}^2+4H^2}}$, and $\mathfrak{a}=\max\limits_{1\le j\le n}|a_j|$. Therefore, $\mathcal{K}_H\to 1$ as $H\to\infty$.

Under the mapping $\rho_H$, the polyline $\gamma_H$ passes into the segment $[ I_{1H}, J_{1H}]$ of the real axis, and the segments $[A_{H}, B_{H}]$ and $[C_{ H}, D_ {H}]$ into segments $[A_{1H}, B_{1H}]$ and $[C_{1H}, D_ {1H}]$ symmetric with respect to the real axis. Here $I_{1H}$, $J_{1H}$, $A_{1H}$, $B_{1H}$, $C_{1H}$ and $D_{1H}$ are the respective images of the points $ I_H $, $J_H$, $A_H$, $B_H$, $C_H$ and $D_H$ under this mapping.  Denote
$$
Q_{1H}=\varrho_H(Q_H)=\{(u, v)\in\mathbb{C}\,:\, aH\le u\le bH, f_1(u/H)\le v\le g_1(u/H)\},
$$
and let
$$
\mbox{\boldmath$Q$}_{1H}=(Q_{1H}; A_{1H}, B_{1H}, C_{1H}, D_{1H}),
$$
where $f_1(x/H)$ and $g_1(x/H)$ are continuous functions on $[aH,bH]$, the graphs of which are the images of the graphs of the functions $y=f(x/H)$ and $y=g(x/H)$ under the mapping $\varrho_H$, respectively.

Since every translation does not change the value of the conformal modulus, without loss of generality we can assume that $b=-a=\alpha>0$, therefore,
$$Q_{1H}=\{(u, v)\in\mathbb{C}\,:\, -\alpha H\le u\le \alpha H, f_1(u/H)\le v\le g_1(u/H)\}.$$

From \eqref{eq:quasi_prop} we get
$$
\frac{1}{\mathcal{K}_H}\operatorname{ExtMod}(\mbox{\boldmath$Q$}_H)\le \operatorname{ExtMod}(\mbox{\boldmath$Q$}_{1H})\le\mathcal{K}_H \operatorname{ExtMod}(\mbox{\boldmath$Q$}_H).
$$
This, together with the fact that $\mathcal{K}_H\to 1$, as $H\to\infty$, implies that
\begin{align}\label{eq:extmodQ1H}
\operatorname{ExtMod}(\mbox{\boldmath$Q$}_{1H})\sim\operatorname{ExtMod}(\mbox{\boldmath$Q$}_H)\quad\text{ as }\quad H\to\infty.
\end{align}

On the other hand, if we denote the quadrilateral $\mbox{\boldmath $G$}_{1H}=(G_{1H}; D_{1H}, C_{1H}, B_{1H}, A_{1H}),$
where $G_{1H}=\overline{\mathbb{C}}\setminus Q_{1H}$ is the complement of $Q_{1H}$ in the extended complex plane, then
\begin{align*}
\operatorname{ExtMod}(\mbox{\boldmath$Q$}_{1H})=\operatorname{Mod}(\mbox{\boldmath$G$}_{1H}).
\end{align*}
This together with \eqref{eq:extmodQ1H} gives
\begin{align}\label{eq:G1HQH}
\operatorname{ExtMod}(\mbox{\boldmath$Q$}_H)\sim\operatorname{Mod}(\mbox{\boldmath$G$}_{1H}) \quad\text{ as }\quad H\to\infty.
\end{align}

We also note that $|A_HB_H|=|A_{1H}B_{1H}|$ and $|D_HC_H|=|D_{1H}C_{1H}|$ for every $H>0$.

Now we consider the following two cases and obtain some upper and lower bounds for $\operatorname{Mod}(\mbox{\boldmath$G$}_{1H})$.

Without loss of generality, consider the case $|A_{1H}B_{1H}|\le |C_{1H}D_{1H}|$, the case $|A_{1H}B_{1H}|\ge |C_{1H}D_{1H}|$ is studied similarly.

It should be emphasized that $f_1(\alpha)=-g_1(\alpha)$ and $f_1(-\alpha)=-g_1(-\alpha)$. For convenience, we denote $\beta=g_1(\alpha)$, $\sigma=g_1(-\alpha)$.

For every $H$, let us fix a sufficiently large number $M>0$ such that
$$-M<f_1(u)<g_1(u)<M, \quad\text{ for all}\quad |u|\le\alpha.$$

Consider the quadrilateral $\mbox{\boldmath$G$}_{2H}:=(G_{2H}; D_{1H}, C_{1H}, B_{1H}, A_{1H})$, where  $G_{2H}$ is the complement of the rectangle $[-\alpha, \alpha]\times[-M, M]$ in the extended complex plane, and note that the vertices of $\mbox{\boldmath$G$}_{2H}$ coincide with the vertices $D_{1H}, C_{1H}, B_{1H}$ and $A_{1H}$ of $\mbox{\boldmath$G$}_{1H}$. It is clear that $G_{2H}\subset G_{1H}$, and by Theorem~\ref{thr:comp2} $(ii)$ we get
\begin{align}\label{mod_up}
\operatorname{Mod}(\mbox{\boldmath$G$}_{1H})\le \operatorname{Mod}(\mbox{\boldmath$G$}_{2H}).
\end{align}

On the other hand, denote by $G_{3H}$ the complement of the union of two vertical segments $[A_{1H}, B_{1H}]$, $[D_{1H}, C_{1H}]$, and the horizontal one, $[I_{1H}, J_{1H}]$. Then, $G_{1H}\subset G_{3H}$, and Theorem~\ref{thr:comp2} $(ii)$ leads to
\begin{align}\label{mod_low}
\operatorname{Mod}(\mbox{\boldmath$G$}_{3H})\le \operatorname{Mod}(\mbox{\boldmath$G$}_{1H}),
\end{align}
where the quadrilateral $\mbox{\boldmath$G$}_{3H}=(G_{3H}; D_{1H}, C_{1H}, B_{1H}, A_{1H})$.

Combining \eqref{eq:G1HQH}, \eqref{mod_up} and \eqref{mod_low}, we obtain
\begin{align}\label{mod_bound}
\operatorname{Mod}(\mbox{\boldmath $G$}_{3H})\le \operatorname{ExtMod}(\mbox{\boldmath$Q$}_H)\le \operatorname{Mod}(\mbox{\boldmath$G$}_{2H}).
\end{align}

If $|A_{1H}B_{1H}|=|C_{1H}D_{1H}|$ then using Lemmas 4.1 and 4.2 from \cite{nguyen2023}, where the case of symmetric quadrilaterals is considered, we can immediately estimate the conformal moduli of the quadrilaterals $\mbox{\boldmath$G$}_{2H}$ and $\mbox{\boldmath$G$}_{3H}$ and show that
$$\operatorname{Mod}(\mbox{\boldmath $G$}_{2H}), \operatorname{Mod}(\mbox{\boldmath $G$}_{3H})\sim\frac{1}{\pi}\log H\quad\text{ as }\quad H\to\infty.$$
This immediately implies the assertion of Theorem~\ref{thr:main}.

In the case of $|A_{1H}B_{1H}|<|C_{1H}D_{1H}|$ we proceed as follows.  Let $\widetilde{A}_{1H}=H\alpha+i\sigma$ and $\widetilde{B}_{1H}=H\alpha-i\sigma$ be the respective perpendicular projections of $A_{1H}$ and $B_{1H}$ on the segment $[D_{1H}, C_{1H}]$. Denote $\mbox{\boldmath$\widetilde{G}$}_{2H}=(G_{2H}; \widetilde{A}_{1H}, \widetilde{B}_{1H}, B_{1H}, A_{1H})$. From  Theorem~\ref{thr:comp2}~$(i)$ we get
\begin{align}\label{mod_up1}
\operatorname{Mod}(\mbox{\boldmath$G$}_{2H})\le\operatorname{Mod}(\mbox{\boldmath$\widetilde{G}$}_{2H}).
\end{align}

Consider the quadrilateral $\widetilde{\mbox{\boldmath$G$}}_{3H}=(G_{3H}; \widetilde{A}^-_{1H}, \widetilde{B}^-_{1H}, B_{1H}, A_{1H})$, where the vertices $\widetilde{A}^-_{1H}, \widetilde{B}^-_{1H}$ are the points which supported by points $\widetilde{A}_{1H}$, $\widetilde{B}_{1H}$, respectively, lying on the left side of the slit along the segment $[D_{1H}, C_{1H}]$, then by Theorem~\ref{thr:comp2} $(i)$, we have
\begin{align}\label{mod_low1}
\operatorname{Mod}(\widetilde{\mbox{\boldmath$G$}}_{3H})\le \operatorname{Mod}(\mbox{\boldmath$G$}_{3H}).
\end{align}

Now let $G^*_{3H}$ be the exterior region of three segments: $[C_{1H},D_{1H}]$, $[\widetilde{C}_{1H},\widetilde{D} _ {1H}]$ and $[I_{1H},J_{1H}]$, where the points $\widetilde{C}_{1H}$ and $\widetilde{D}_{1H}$ are respective symmetric to the points $C_{ 1H}$ and $D_{1H}$ with respect to the imaginary axis. Consider the quadrilateral $\mbox{\boldmath$G$}^*_{3H}=(G^*_{3H}; \widetilde{A}^-_{1H}, \widetilde{B}^-_{1H }, B^+_{1H}, A^+_{1H})$, where $A^+_{1H}$, $B^+_{1H}$ are respective supported by points $A_{1H}, B_{1H}$, and lying on the right side of the slit along the segment $[A_{1H}, B_{1H}]$, then from Theorem~\ref{thr:comp2} $(i)$, we obtain
\begin{align}\label{mod_low2}
\operatorname{Mod}(\mbox{\boldmath$G$}^*_{3H})\le \operatorname{Mod}(\widetilde{\mbox{\boldmath$G$}}_{3H}).
\end{align}

Combining \eqref{mod_bound}, \eqref{mod_up1}, \eqref{mod_low1} and \eqref{mod_low2}, we deduce that
\begin{align}\label{mod_bound1}
\operatorname{Mod}(\mbox{\boldmath$G$}^*_{3H})\le \operatorname{ExtMod}(\mbox{\boldmath$Q$}_H)\le \operatorname{Mod}(\widetilde{\mbox{\boldmath$G$}}_{2H}).
\end{align}

Now, using the ideas of the proofs of Lemmas~4.1--4.2 from \cite{nguyen2023}, we study the asymptotic behavior of the conformal moduli of the quadrilaterals $\mbox{\boldmath$\widetilde{G}$}_{2H}$ and  $\mbox{\boldmath$G$}^*_{3H}$ as $H\to\infty$.

\begin{lemma}\label{lem_mod_up}
We have
$$
\operatorname{Mod}(\mbox{\boldmath$\widetilde{G}$}_{2H})\sim\frac{1}{\pi}\,\log H\quad\text{ as }\quad H\to\infty.
$$
\end{lemma}
\begin{proof}
This assertion is actually established in \cite[Lemma~4.1]{nguyen2023}.
\end{proof}

\begin{lemma}\label{lem_mod_low}
We have
\begin{align}\label{bound_low}
\operatorname{Mod}(\mbox{\boldmath$G$}^*_{3H})\sim\frac{1}{\pi}\,\log H\quad\text{ as }\quad H\to\infty.
\end{align}
\end{lemma}
\begin{proof}
Let $\widetilde{D}_{1H}=-H\alpha+i\beta, \widetilde{C}_{1H}=-H\alpha-i\beta$ be the perpendicular projections of $D_{1H}, C_{1H}$ on the extension of the segment $[A_{1H}, B_{1H}]$, respectively. The conformal mapping of the exterior of two horizontal segments, $[-1/k,-1]$ and $[1,1/k]$ ($0<k<1$) onto the exterior of two vertical segments $[\widetilde{D}_{1H}, \widetilde{C}_{1H}]$ and $[D_{1H}, C_{1H}]$ has the form
\begin{equation}\label{zetaw1}
\zeta(w) =C \int\limits_0^w \frac{\left(1/\lambda^2-t^2\right) dt}{\sqrt{(1 - t^2)(1 - k^2 t^2)}}\,.
\end{equation}
Here $C\ne 0$ is the some constant, and $1/\lambda \in(1, 1/k)$,  $1/\mu\in(1, 1/\lambda)$ are respective the preimages of the points $D_{1H}$ and $\widetilde{A}^-_{1H}$.

Now let us describe the detailed formula of $\zeta(w)$, from \eqref{zetaw1} we have
\begin{align*}
\zeta(w)&=\frac{C}{k^2}\int\limits_0^w \frac{\left(k^2/\lambda^2-k^2t^2\right)dt}{\sqrt{(1 - t^2)(1 - k^2 t^2)}}=\frac{C}{k^2}\int\limits_0^w \frac{\left(1-k^2t^2\right)-\left(1-k^2/\lambda^2\right)dt}{\sqrt{(1 - t^2)(1 - k^2 t^2)}}\\
&=\frac{C}{k^2}\left[\int\limits_0^w \sqrt{\frac{1-k^2t^2}{1-t^2}}dt-\left(1-k^2/\lambda^2\right)\int_0^w\frac{dt}{\sqrt{(1 - t^2)(1 - k^2 t^2)}}\right]\\
&=\frac{C}{k^2}\left[E(w,k)-(1-k^2/\lambda^2)F(w,k) \right],
\end{align*}
where $K(w, k)$ and $E(w, k)$ are the incomplete elliptic integrals of the first and second kinds, respectively. By the definition of $\zeta(w)$, we imply that
$$
\zeta(1)=H\alpha,\qquad\zeta(1/\lambda)=H\alpha+i\beta,\qquad\zeta(1/\mu)=H\alpha+i\sigma.
$$
It is not difficult to see that (see, e.g., \cite{byrd1971handbook})
\begin{align}\label{Hakap}
H\alpha=\frac{C}{k^2} \left[E(k)-(1-k^2/\lambda^2)K(k) \right],\quad \beta =\frac{C}{k^2}[E(\ell',k')-(k^2/\lambda^2)F(\ell',k')],
\end{align}
where
\begin{align}\label{ell}
\ell'=\sqrt{1-\ell^2}, \quad \ell=\frac{k}{k'}\,\sqrt{\frac{1}{\lambda^2} -1}, \quad \frac{1}{\lambda^2}=\frac{1}{k^2} \frac{E'(k)}{K'(k)}.
\end{align}
It should be emphasized that for every $H>0$ there are unique $k=k(H)$, $\lambda=\lambda(H)$, $\ell=\ell(H)$ and $C=C(H)$ such that \eqref{zetaw1}, \eqref{Hakap} and \eqref{ell} take place, if $\alpha$, $\beta$ are fixed; moreover, $k=k(H)\to 1$, $\lambda=\lambda(H)\to 1$ as $H\to\infty$. Since all the parameters in \eqref{zetaw1} are uniquely defined by $H$, we will write $\zeta(w)=\zeta_H(w)$.

By the Riemann-Schwarz symmetry principle, $\zeta_H(w)$  maps the upper half-plane $\mathbb{H}_{\zeta}^+$ onto $G_{5H}$ which is the exterior of the union of three segments $[\widetilde{B}_{1H}, \widetilde{C}_{1H}]$, $[B_{1H}, C_{1H}]$ and $[I_{1H}, J_{1H}]$; in the following we will only consider $\zeta_H(w)$ on $\mathbb{H}_{\zeta}^+$.

From \eqref{Hakap} it implies that
\begin{align}\label{Ha/sig}
\frac{E(k)-(1-k^2/\lambda^2)K(k)}{E(\ell',k')-(k^2/\lambda^2) F(\ell',k')}=\frac{H\alpha}{\beta}.
\end{align}

Now we will show that
\begin{align}\label{muk}
\frac{1-\lambda(H)}{1-k(H)}\to\delta_1\in (0,1) \text{ and }\frac{1-\mu(H)}{1-k(H)}\to\delta_2\in (0,1) \quad\text{ as }\quad  H\to\infty.
\end{align}

Indeed, for every $H>0$ the function $\psi_H=\zeta_H(1+(1/k-1)z)-\alpha H$ maps the upper half-plane onto the domain $G_H$ which is the half-plane with two slit cuts along the segments, $[0,0+i\beta]$ and $[-2H\alpha,-2H\alpha +i\beta]$. In addition, $\psi_H$ maps the point $0$, $1$, and $\infty$ to $0^-_H$, $0^+_H$, and $\infty$, where $0^-_H$ and $0^+_H$ are the prime ends of $G_H$ with support at the origin, lying on the left and right edges of the slit along the segment $[0,0+i\beta]$. As $H\to\infty$, the family of domains $G_H$  converges to $G$ as a kernel which is the upper half-plane with a slit cut along the segment $[0, 0+i\beta]$. Then, by the generalized Rad\'o theorem, $\psi_H$ converges uniformly to  the conformal mapping $\psi$ of the upper half-plane onto $G$ such that $0$, $1$, and $\infty$ are mapped to $0^-$, $0^+$, and $\infty$, where $0^-$ and $0^+$ are the prime ends of $G$ with support at $0$, lying on the left and right edges of the slit along the segment $[0,0+i\beta]$. By the definition of $\psi_H$, we have $\psi_H((1/\lambda(H)-1)/(1/k(H)-1))=i\beta$, this deduces that $(1/\lambda(H)-1)/(1/k(H)-1)\to\delta_1$, where $\delta_1$ is the preimage of $i\beta$ under $\psi$; it is similar, we also have $\psi_H((1/\mu(H)-1)/(1/k(H)-1))=(i\sigma)^-$, and  $(1/\mu(H)-1)/(1/k(H)-1)\to\delta_2$, where $\delta_2$ is the preimage of $(i\sigma)^-$ under $\psi$. Therefore, $0<\delta_2<\delta_1<1$ and \eqref{muk} holds.

Combining \eqref{ell} and \eqref{muk} results in $\ell'=\ell'(H)\to\sqrt{\delta_1}\in(0, 1)$.

With the note that $k<\lambda<1$, from \eqref{asymke0} and \eqref{asymke1}, for $k=k(H)$ and $\lambda=\lambda(H)$, we deduce that 
\begin{align}\label{nenoHa/sig}
E(k)-(1-k^2/\lambda^2)K(k)\to 1 \quad\text{as} \quad H\to \infty.
\end{align}

From \eqref{ell} it follows that
$$
\frac{k^2}{\lambda^2}=\frac{E'(k)}{K'(k)}.
$$
Thus, for $k=k(H)$, $\lambda=\lambda(H)$ and $\ell=\ell(H)$ we have
\begin{align*}
E(\ell', k')-(k^2/\lambda^2)F(\ell', k')&=\frac{K(k')E(\ell', k')-E(k')F(\ell', k')}{K(k')}\\&\sim\frac{2}{\pi}\left[ K(k')E(\ell', k')-E(k')F(\ell', k')\right]\quad\text{ as }\quad H\to \infty.
\end{align*}
In \eqref{eq:exp_int1} and \eqref{eq:exp_int2}, replacing the parameters $x$ and $k$ with $\ell'$ and $k'$, we get
$$
E(\ell', k')=I_0(\ell')-\frac{1}{2}I_2(\ell')k'^2-\frac{1}{8}I_4(\ell')k'^4-\ldots\,,
$$
$$
F(\ell', k')=I_0(\ell')+\frac{1}{2}I_2(\ell')k'^2+\frac{3}{8}I_4(\ell')k'^4+\ldots\, .
$$
Combined with \eqref{Ipi} it shows that
$$
E(\ell', k')=I_0(\ell')-\frac{1}{2}I_2(\ell')k'^2+o(k'^2),
$$
$$
F(\ell', k')=I_0(\ell')+\frac{1}{2}I_2(\ell')k'^2+o(k'^2),
$$
uniformly with respect to $\ell'$ as $k'\to 0$.

This together with \eqref{eq:com_int1_serie1} and \eqref{eq:com_int2_serie1} implies
$$
K(k')E(\ell', k')-E(k')F(\ell', k') =\frac{\pi}{4}\left[I_0(\ell')-2I_2(\ell')\right]k'^2+ o(k'^2)
$$
uniformly with respect to $\ell'$ as $k'\to 0$.  Since $\ell'=\ell'(H)\to \sqrt{\delta_1}\in(0,1)$ as $H\to\infty$, we have
\begin{align}\label{denoHa/sig}
K(k')E(\ell', k')-E(k')F(\ell', k') \sim\frac{\pi}{4}\left[I_0(\sqrt{\delta_1})-2I_2(\sqrt{\delta_1})\right]k'^2,
\end{align}
for $k'=k'(H)$, $\ell'=\ell'(H)$ as $H\to \infty$. 

Since \eqref{I02} then $I_0(\sqrt{\delta_1})-2I_2(\sqrt{\delta_1})\neq 0$. From \eqref{Ha/sig}, \eqref{nenoHa/sig} and \eqref{denoHa/sig} we obtain for $k'=k'(H)$
$$
\frac{H\alpha}{\beta}\sim\frac{2C_1}{k'^2} \quad\text{ as }\quad H\to\infty,
$$
with some constant $C_1\neq 0$.
This implies that
\begin{align}\label{asym_H}
\frac{1}{\pi}\log H\sim\frac{1}{\pi}\log\frac{1}{1-k}
\quad\text{ as \quad $k\to 1$}.
\end{align}

Consider the quadrilateral $\mbox{\boldmath$G$}^{*+}_{3H}=(G^{*+}_{3H}; \widetilde{A}^-_{1H}, J_{1H}, I_{1H}, A^+_{1H})$ where $G^{*+}_{3H}$ is the upper half of $G^*_{3H}$. Since it is conformally equivalent to the quadrilateral which is the upper half-plane with vertices $-1/k$, $-1/\mu$, $1/\mu$ and $1/k$,
 we conclude  that
$$
\operatorname{Mod}(\mbox{\boldmath$G$}^{*+}_{3H})= \frac{2K(k/\mu)}{K'(k/\mu)}.
$$

With the help of the symmetry principle, we see that
$$\operatorname{Mod}(\mbox{\boldmath$G$}^*_{3H})=\frac{K(k/\mu)}{K'(k/\mu)}\sim \frac{1}{\pi}\log\frac{1}{1-k/\mu}, \quad k=k(H),\quad \mu=\mu(H)\quad\text{ as }\quad  H\to\infty.
$$

As $H\to\infty$, the following estimates are going
$$1-\frac{k(H)}{\mu(H)}\sim\mu(H)-k(H)=(1-k(H))-(1-\mu(H))\sim(1-\delta_2)(1-k(H)),$$
therefore,
$$
\frac{1}{\pi}\log\frac{1}{1-k/\mu}\sim \frac{1}{\pi}\log\frac{1}{1-k},
$$
and, this along with \eqref{asym_H} implies \eqref{bound_low}.
\end{proof}

\textit{Proof of the main theorem.}  From Lemma~\ref{lem_mod_up}, Lemma~\ref{lem_mod_low} and the inequality \eqref{mod_bound1}, we conclude that the asymptotic formula \eqref{eq:main} is satisfied. Therefore, Theorem~\ref{thr:main} has been fully verified.

\section*{FUNDING}
The work of the first author is performed under  the Development Program of the Scientific and Educational Mathematical Center of the Volga Region (agreement No.~075-02-2022-882).

%%%%%%%%%%%%%%%%%%

\end{document}